\documentclass{article}
\usepackage{geometry}

\usepackage{blindtext} 
\usepackage{amssymb}
\usepackage{amsmath, amsthm, amssymb}
\usepackage{multicol}
\usepackage{bbm}
\usepackage{cite}
\usepackage{dsfont}

\newcommand{\E}{\mathbb{E}}
\renewcommand{\P}{\mathbb{P}}

\newtheorem{prop}{Proposition}
\newtheorem{thm}{Theorem}
\newtheorem{lemma}{Lemma}
\newtheorem{corollary}{Corollary}

\newtheorem{definition}{Definition}

\begin{document}

\title{The Oldest Allele in a Moran Model}
\author{Douglas Rizzolo \\ University of Delaware \and Marydol Soto-Santarriaga \\ University of Delaware}

\maketitle

\begin{abstract}
We study an infinitely many neutral alleles Moran model where the allelic frequencies are recorded in increasing order of the ages of the alleles.  Using this representation, we show that in stationarity the number of individuals with the oldest allelic type evolves as a Markov process, even across times when the oldest allele goes extinct (and is immediately replaced by the formerly second oldest allele).  We also find the diffusive limit of the oldest allele chain.
\end{abstract}

\section{Introduction}
Studying the relationship between allele ages and allele frequency has a long history in population genetics, going back at least to the work of Kimura and Ohta \cite{kimura1973age}. In equilibrium, the composition of alleles ordered by age has been studied extensively, see e.g \cite{donnelly1986ages} and the references therein.  One challenge in studying the ages of alleles is that this information is not included in the state space of many common models such as the Moran model.   To address this, measure valued process incorporating ages were studied \cite{EKFlemingViotwAges}.   Despite the classical nature of problems in this area, they continue to be actively studied.  For example, \cite{malaspinas2012estimating} developed a likelihood method to jointly estimate the selection coefficient and the age of an allele from time-serial data. Meanwhile, \cite{slatkin2000allele} obtained a simple maximum-likelihood estimator of allele age based on frequency. Recently \cite{balocchi2024bayesian}, focused on developing a Bayesian Nonparametric (BNP) approach to a species sampling problem (SSPs) using an ordered Chinese Restaurant process from \cite{JamesNTR}.


Our motivation is to study the forward in time, dynamic, evolution of the allelic composition of a population ordered by age. In particular, we study an integer composition valued variant of the infinitely many neutral alleles Moran model where mutations always produce new alleles that are placed as the left-most block of the composition.  In this way, if the allelic composition is initially ordered left-to-right from youngest to oldest, it will remain in this order for all time.  We discuss the diffusive limit of this process, which in an interval-partition valued diffusion.  Our main application is to show that, in equilibrium, the frequency of the oldest allele evolves as a Markov process for all time (including across times where the oldest allele vanishes from the population and is immediately replaced by the formerly second oldest allele).  The dynamics we study are close to other classical infinitely-many neutral alleles models are the works of \cite{ewens1972sampling, watterson1976reversibility, ethier1981infinitely}.  By working with composition valued processes, we are able to study the relative ages of alleles without going to the full measure-valued process framework that includes the specific ages of alleles.  Our main result shows that, in stationarity, the number of individuals with the oldest allelic type evolves as a Markov process, even across times when the oldest allele goes extinct (and is immediately replaced by the formerly second oldest allele).  We also find the diffusive limit of the oldest allele chain.

\subsection{An Infinitely Many Neutral Alleles Moran Model}

Consider a population of constant size, $M$ individuals, each of which could be any one of an infinite sequence of possible allelic types. Moran \cite{moran1958random} introduced a model in which a birth and a death occurred simultaneously, thus preserving the constant size of the population but only changing it by two individuals at a time. The birth is carried out by a random choice of a parent from all $M$ individuals present. For some $U$ between zero and one, the parent's type is passed on to the offspring with probability $1-U$, or, with probability $U$, the offspring will be of a mutant type not previously present in the population. The death also occurs at random from among the $M$ individuals, excluding the offspring just born. Some allelic types may be represented by more individuals than others.

Watterson \cite{watterson1976reversibility} introduces to Moran's model the association of ages to the allelic types that the population presents, where a newly born mutant allelic type has age 1. For each subsequent birth-death event during which that allelic type survive, they add an extra 1 to the age of that allele. They then consider the probability distribution of the age of the allele, that is, the time that has elapsed between the introduction of the allele by mutation and the present. 
Similarly to Moran's work, we consider a model with a birth and a death occurring where mutation can happen. However, we shall consider a model that does not exclude the new offspring during the death process, and as Wattersons does, we also care about describing the ages of the allelic types in our population at a time. 

In this paper, we consider the following dynamics.  At each time an individual is chosen uniformly at random to reproduce.  They produce an individual of their own allelic type with probability $1/(n+\theta)$ or a mutation occurs and they produce an individual with a new allelic type with probability $\theta/(n+\theta)$, where $\theta>0$ is a mutation parameter.  After the new individual is born, an individual is chosen uniformly at random (including the new individual) and removed from the population.  The particular choice of these dynamics is so that the resulting chain on compositions falls within the framework of two-parameter ordered Chinese Restaurant Process up-down chains \cite{rivera2023diffusive}.  We remark that our ordered Chinese Restaurant Process comes from the regenerative structure of the compositions involved \cite{GPRegen}, unlike \cite{balocchi2024bayesian}, which uses the ordered Chinese Restaurant Process coming from the EPPF.

\begin{definition} For $n\geq 1$, a composition of $n$ is a tuple $\sigma= (\sigma_1,...,\sigma_k)$ of positive
integers that sum to $n$. The composition of $n = 0$ is the empty tuple, which we denote by $\varnothing$. If $\sigma$ is a composition of $n$ with $k$ components, we say it has size
$|\sigma|= n$ and length  $l(\sigma) = k$. We denote the set of all compositions of $n$ by $C_n$ and their union by $C$.
\end{definition}

An up-down chain on $C_n$ is a Markov chain whose steps can be factored into two parts: 1) an up-step from $C_n$ to $C_{n+1}$ according to a kernel $p^{\uparrow}$ followed by 2)
a down-step from $C_{n+1}$ to $C_n$ given by a kernel $p^\downarrow$. We denote by $T_n(\sigma,\sigma')$ the probability of transitioning from $\sigma$ to $\sigma'$ that is given by
$$T_n(\sigma,\sigma')=\sum\limits_{\tau \in C_{n+1}}p^{\uparrow}(\sigma,\tau)p^\downarrow(\tau,\sigma').$$

In the up-down chain we consider, the up-step kernel $p^\uparrow_{(0,\theta)}$ given by an $(0,\theta)$-ordered Chinese Restaurant Process growth step. In the Chinese Restaurant
Process analogy, we view the initial composition $\tau=(\tau_1,\dots,\tau_k) \in C_n$ as a numbered list of customers at $k$ occupied tables in a restaurant. In our setting, we think of customers as the alleles and the tables as the allelic types, so that $\tau_i$ is the number of alleles of type $i$ on the list and $\theta$ represents our mutation parameter. An up-step from $\tau$ then corresponds to the birth of a new allele that is assign its allelic type according to the following rules:
\begin{itemize}
    \item The new allele has the allelic type $i$ with probability $\tau_i/(n+\theta),$ resulting in a step from $\tau$ to $(\tau_1,\dots,\tau_{i-1},\tau_i+1,\dots,\tau_k).$
    \item The new allele has a new allelic type with probability $\theta/(n+\theta).$ To preserve the order of the ages of allele types, the new type appear at the beginning of the list, resulting in a step from $\tau$ to $(1,\tau_1,\tau_2,\dots,\tau_k).$ Where we go from left to right in order of youngest to oldest allelic type. 
\end{itemize}

The down-step kernel $p^\downarrow$ can be thought as describing an allele dying. So, if we view an initial composition $\tau \in C_{n+1}$ as describing the arrangement of the alleles by their types from youngest to oldest. A down-step from $\tau$ then corresponds to a uniformly random allele dying. Hence,
\begin{itemize}
    \item an allele of type $i$ dies with probability $\tau_i/(n+1),$ resulting in a step from $\tau$ to 
\end{itemize}
\[
\begin{cases}
(\tau_1,\dots,\tau_{i-1},\tau_i-1,\tau_{i+1},\dots,\tau_k) & \text{if } \tau_i > 1 \\
(\tau_1,\dots,\tau_{i-1},\tau_{i+1},\dots,\tau_k) & \text{if } \tau_i = 1 
\end{cases}
\]

It is clear then that the rightmost column of the resulting tuple represents the oldest allelic type in the population at that time. In general, the ages of the allelic types are ordered from oldest to youngest when seen from right to left. With this setting, we can then study the oldest allele type, and more generally we can consider the k-th oldest allelic types in our population.

We are working with an analogue case to the one considered in \cite{rivera2022leftmost}, where they have an up-down chain with an up-step kernel $p^\uparrow_{(\alpha,0)}$ given by an $(\alpha,0)$-ordered Chinese Restaurant Process growth step and a uniform down step. They study the behavior of the leftmost column, but unlike the $(0,\theta)$ case we study, the $(\alpha,0)$ case does not have a natural population genetics interpretation.  It is important to note that contrary to the usual convention of ordering oldest to youngest from left to right, which can be seen in \cite{balocchi2024bayesian} or \cite{ethier1981infinitely}, our model is the left-to-right reversal of that ordering.

Define
\[T^{(0,\theta)}_n(\sigma,\sigma')=\sum\limits_{\tau \in C_{n+1}}p_{(0,\theta)}^{\uparrow}(\sigma,\tau)p^\downarrow(\tau,\sigma').\]
Let $(X_n^{(0,\theta)}(k))_{k\geq{0}}$ be the Markov chain in $C_n$ with the transition kernel $T_n^{(0,\theta)}$.  The projection $\psi(\sigma)=\sigma_{l(\sigma)}$ for $\sigma\neq\varnothing$ gives rise to the rightmost column processes defined by $Z_n^{(0,\theta)}=\psi(X_n^{(0,\theta)})$.

Our two main results are the following.

\begin{thm}
If $(X_n^{(0,\theta)}(k))_{k\geq{0}}$ is running in stationarity then $Z_n^{(0,\theta)}$ is a Markov chain.
\end{thm}

For a more precise statement, see Proposition \ref{prop:rightmost}, which also provides the transition matrix of $Z_n^{(0,\theta)}$.  Indeed, we can also prove that the vector of numbers of individuals with each of the $k$ oldest alleles is a Markov process, see Proposition \ref{propkoldest}.

\begin{thm}\label{thmscaling}
Let $\alpha_n = (n+\theta)(n+1)$.  If $(X_n^{(0,\theta)}(k))_{k\geq{0}}$ is running in stationarity then 
\[
\left( \frac{1}{n} Z_n^{(0,\theta)}(\lfloor \alpha_n t \rfloor), t \ge 0\right) \Rightarrow Z^{(0,\theta)}
\] 
in the Skorokhod topology where $Z^{(0,\theta)}$ is a diffusion on $[0,1]$ jumping in from the boundary at $0$ with generator $\mathcal{L}:\mathcal{D}\subseteq C[0,1]\xrightarrow{}C[0,1]$ given by
       $$\mathcal{L}f(x)=x(1-x)f''(x)-\theta xf'(x)$$
       for $x\in (0,1),$ where $\mathcal{D}$ forms a core for the true infinitesimal generator and consists of functions $f \in C^2([0,1])$ satisfying 
       \begin{enumerate}
           \item $\int_{0}^{1} (f(x)-f(0))\theta(1-x)^{\theta-1} \,dx=0,$ and
           \item $\lim_{x\rightarrow1}f'(x)(1-x)^\theta = 0$ if $\theta < 1$.
       \end{enumerate}
\end{thm}

The reason there is only a boundary condition at $1$ if $\theta <1$ is because if $\theta<1$ then the boundary at $1$ is accessible from $(0,1)$, and in this case the condition is that the boundary is reflecting.  If $\theta\geq 1$ the boundary is inaccessible from $(0,1)$ and thus no boundary condition is needed.  The boundary condition at $0$ means that when the process hits $0$ it instantaneously jumps into $(0,1)$ and the density of the location it jumps to is $p(x) = \theta(1-x)^{\theta-1} $.  Because $p$ is a density (and in particular is integrable), the process $Z^{(0,\theta)}$ is not a Feller process, which means that Theorem \ref{thmscaling} cannot be proved using the standard convergence of generators arguments.  Instead, our argument goes though excursion theory by considering the pieces of the path between zeros.

\section{The Oldest Allele}
As above, let $(X_n^{(0,\theta)}(k))_{k\geq{0}}$ be the Markov chain in $C_n$ with the transition kernel $T_n^{(0,\theta)}$ as defined in \cite{rivera2023diffusive}. Let $M_n^{(0,\theta)}$ denote the unique stationary distribution of $(X_n^{(0,\theta)}(k))_{k\geq{0}}$. 
The projection $\psi(\sigma)=\sigma_{l(\sigma)}$ for $\sigma\neq\varnothing$ gives rise to the rightmost column processes defined by $Z_n^{(0,\theta)}=\psi(X_n^{(0,\theta)})$. Let $\upsilon_n^{(0,\theta)}=M_n^{(0,\theta)}\circ \psi^{-1}$, which tracks the number of individuals with the oldest allele.

It is clear that for $n \geq i\geq 1$ and $\sigma \in C_{n-i}$ we have that 
\begin{equation} \label{eq:1}
M_n^{(0,\theta)}(\sigma,i)=M_{n-i}^{(0,\theta)}(\sigma)\upsilon_n^{(0,\theta)}(i).
\end{equation} 

Moreover, the following condition holds for $n\geq1$:
\begin{equation} \label{eq:2}
    M_n^{(0,\theta)}=M_{n-1}^{(0,\theta)}p^\uparrow_{(0,\theta)}=M_{n+1}^{(0,\theta)}p^\downarrow
\end{equation}
since we are a special case with $\alpha=0$ of \cite[Proposition 1.1]{rivera2023diffusive}.
Consider taking an $(0,\theta)$ up-step from $(\sigma,i)$ followed by a down-step. Let $U$ be the event in which this up-step stacks a box in the last column of $(\sigma,i)$, and let $D$ be the event in which the down-step removes a box from the last column of a composition. Let $r_{i,i+1}=\mathbb{P}(U\cap D^c)$, $r_{i,i-1}=\mathbb{P}(U^c\cap D)$, $r_{i,i}^{(1)}=\mathbb{P}(U^c\cap D^c)$, $r_{i,i}^{(2)}=\mathbb{P}(U\cap D)$, and $r_{i,i}=r_{i,i}^{(1)}+r_{i,i}^{(2)}$. We have the following formulas

\begin{align*}
    r_{i,i-1}&=\frac{(n+\theta-i)i}{(n+\theta)(n+1)}, & r_{i,i+1}&=\frac{(n-i)i}{(n+\theta)(n+1)},\\
    r_{i,i}^{(1)}&=\frac{(n+\theta-i)(n+1-i)}{(n+\theta)(n+1)}, &  r_{i,i}^{(2)}&=\frac{(i+1)i}{(n+\theta)(n+1)}.\\
\end{align*}

The probability of these events does not depend on $\sigma.$ We use these events to find an identity relating the transition kernel of the $(0,\theta)$ chain. 

\begin{prop}\label{prop:Tform}
For $n\geq1$ and $(\sigma,i),(\sigma',j) \in C_n$, we have the following identity

\begin{align*} 
T_n^{(0,\theta)} ((\sigma,i),(\sigma',j))&= r_{i,i}^{(2)} \mathbbm{1} (\sigma=\sigma') \mathbbm{1} (i=j)+ r_{i,j}\mathbbm{1}(j=i+1)p^\downarrow(\sigma,\sigma')\\
&+r_{i,j}\mathbbm{1}(j=i-1)p^\uparrow_{(0,\theta)}(\sigma,\sigma')\\
&+r_{1,0}\mathbbm{1}(i=1)p^\uparrow_{(0,\theta)}(\sigma,(\sigma',j))\\
&+r_{i,i}^{(1)}\mathbbm{1}(j=i)T_{n-i}^{(0,\theta)}(\sigma,\sigma')
\end{align*}
\end{prop}

\begin{proof}
    Given a composition $\tau=(\tau_1,\tau_2,...,\tau_{l(\tau)})$, let $\tau_1^{l(\tau)-1}=(\tau_1,\tau_2,...,\tau_{l(\tau)-1})$ be the composition obtained by removing the last column of $\tau$.  Fix $(\sigma,i),(\sigma',j) \in C_n$. Let $C^\uparrow$ be the composition obtained by performing an $(0,\theta)$ up-step from $(\sigma,i)$ and $C^\downarrow$ be the composition obtained by performing a down-step from $C^\uparrow$.  We compute the following conditional probabilities to obtain the identity. 

    Condition on $U\cap D$.
    \begin{align*} 
    \mathbb{P}(C^\downarrow=(\sigma',j)|U, D)&=\mathbb{P}(C^\downarrow=(\sigma',j)|C^\uparrow=(\sigma,i+1), D)\\
    &=\mathbbm{1}((\sigma',j)=(\sigma,i))\\
    &=\mathbbm{1}(\sigma'=\sigma)\mathbbm{1}(j=i).
    \end{align*}
    
    We continue by conditioning on $U\cap D^c$.
    \begin{align*} 
    \mathbb{P}(C^\downarrow=(\sigma',j)|U, D^c)&=\mathbb{P}(C^\downarrow=(\sigma',j)|C^\uparrow=(\sigma,i+1), D^c)\\
    &=\mathbb{P}((C^\downarrow)_1^{l(C^\downarrow)-1}=\sigma',C^\downarrow_{l(C^\uparrow)}=j|C^\uparrow=(\sigma,i+1), D^c)\\
    &=\mathbbm{1}(j=i+1)\mathbb{P}((C^\downarrow)_1^{l(C^\downarrow)-1}=\sigma'|(C^\uparrow)_1^{l(C^\uparrow)-1}=\sigma, D^c)\\
    &=\mathbbm{1}(j=i+1)p^\downarrow(\sigma,\sigma').
    \end{align*}
    
    Next we condition on $U^c\cap D^c$.
    \begin{align*}
    &\mathbb{P}(C^\downarrow=(\sigma',j)|U^c, D^c)=\mathbbm{1}(j=i)\mathbb{P}((C^\downarrow)_1^{l(C^\downarrow)-1}=\sigma'|U^c, D^c)\\
    &=\mathbbm{1}(j=i)\sum\limits_{\tau\in C_{n+1-i}} \mathbb{P}((C^\uparrow)_1^{l(C^\uparrow)-1}=\tau|U^c, D^c) \mathbb{P}((C^\downarrow)_1^{l(C^\downarrow)-1}=\sigma'|C^\uparrow=(\tau,i),D^c)\\
    &=\mathbbm{1}(j=i)\sum\limits_{\tau\in C_{n+1-i}} \mathbb{P}((C^\uparrow)_1^{l(C^\uparrow)-1}=\tau|U^c) \mathbb{P}((C^\downarrow)_1^{l(C^\downarrow)-1}=\sigma'|(C^\uparrow)_1^{l(C^\uparrow)-1}=\tau,D^c)\\
    &=\mathbbm{1}(j=i)\sum\limits_{\tau\in C_{n+1-i}}p^\uparrow_{(0,\theta)}(\sigma,\tau)p^\downarrow(\tau,\sigma')\\
    &=\mathbbm{1}(j=i)T_{n-i}^{(0,\theta)}(\sigma,\sigma').
     \end{align*}
    
    Finally, we condition on  $U^c\cap D$. It is important to note that when $i=1$ the rightmost column disappears when the event $U^c\cap D$ holds. So, we need to take into account the following two cases:
    
    Case 1. $i>1$
    \begin{align*}
    \mathbb{P}(C^\downarrow=(\sigma',j)|U^c, D)&=\mathbb{P}(C^\downarrow=(\sigma',j)|C^\uparrow_{l(C^\uparrow)}=i, D)\\
    &=\mathbbm{1}(j=i-1)\mathbb{P}(\sigma'=(C^\uparrow)_1^{l(C^\uparrow)-1}|U^c)\\
    &=\mathbbm{1}(j=i-1)p^\uparrow_{(0,\theta)}(\sigma,\sigma').
    \end{align*}    
    
    Case 2. $i=1$
    \begin{align*}
    \mathbb{P}(C^\downarrow=(\sigma',j)|U^c, D)&=\mathbb{P}(C^\downarrow=(\sigma',j)|C^\uparrow_{l(C^\uparrow)}=1, D)\\
    &=\mathbb{P}((C^\uparrow)_1^{l(C^\uparrow)-1}=(\sigma',j)|U^c)\\
    &=p^\uparrow_{(0,\theta)}(\sigma,(\sigma',j)).
    \end{align*}

    Combining all of the conditional probabilities terms we get our identity.
\end{proof}

Let $n\geq1$. We define a transition kernel $\Lambda_n$ from $[n]$ to $C_n$ by
$$\Lambda_n(i,(\sigma,i))=M_{n-i}^{(0,\theta)}(\sigma)$$
Let $\psi_n(\sigma,i)=\mathbbm{1}(\sigma_{l(\sigma)}=i)$ be a transition kernel from $C_n$ to $[n]$. 

\begin{prop}\label{prop:rightmost}
    For $n\geq1$, the transition kernel $Q_n^{(0,\theta)}=\Lambda_nT_n^{(0,\theta)}\psi_n$ satisfies 
    $$\Lambda_nT_n^{(0,\theta)}=Q_n^{(0,\theta)}\Lambda_n.$$
    Moreover, $Z_n^{(0,\theta)}$ is a time-homogeneous Markov chain with transition kernel $Q_n^{(0,\theta)}$ given explicitly by 
    $$Q_n^{(0,\theta)}(i,j)=r_{i,j}+r_{1,0}\mathbbm{1}(i=1)\upsilon_n^{(0,\theta)}(j)$$
    when the initial distribution of $X_n^{(0,\theta)}$ is of the form $\mu\Lambda_n$.
\end{prop}
\begin{proof}
    Let $C_n(i,j)=r_{i,j}+r_{1,0}\mathbbm{1}(i=1)\upsilon_n^{(0,\theta)}(j)$. Fix $i,j\in[n]$ and $\sigma'\in C_{n-j}$. We proceed by computing 
    \begin{align*}
        \Lambda_nT_n^{(0,\theta)}(i,(\sigma',j))&=\sum\limits_{\sigma\in C_{n-i}}\Lambda_n(i,(\sigma,i))T_n^{(0,\theta)}((\sigma,i),(\sigma',j))\\
        &=\sum\limits_{\sigma\in C_{n-i}}M_{n-i}^{(0,\theta)}(\sigma)T_n^{(0,\theta)}((\sigma,i),(\sigma',j))\\
        &=\sum\limits_{\sigma\in C_{n-i}}M_{n-i}^{(0,\theta)}(\sigma)[(r_{i,j}^{(2)}\mathbbm{1}(\sigma=\sigma')+r_{i,i}^{(1)}T_{n-i}^{(0,\theta)}(\sigma,\sigma'))\mathbbm{1}(j=i)]\\
        &+\sum\limits_{\sigma\in C_{n-i}}M_{n-i}^{(0,\theta)}(\sigma)r_{i,j}[\mathbbm{1}(j=i+1)p^\downarrow(\sigma,\sigma')+\mathbbm{1}(j=i-1))p^\uparrow_{(0,\theta)}(\sigma,\sigma')]\\
        &+\sum\limits_{\sigma\in C_{n-i}}M_{n-i}^{(0,\theta)}(\sigma)r_{1,0}\mathbbm{1}(i=1)p^\uparrow_{(0,\theta)}(\sigma,(\sigma',j))
\end{align*}

From Equation (2) we can simplify to 
\begin{align*}
        \Lambda_nT_n^{(0,\theta)}(i,(\sigma',j))&=\mathbbm{1}(j=i)[r_{i,i}^{(1)}M_{n-j}^{(0,\theta)}(\sigma')+r_{i,i}^{(2)}M_{n-j}^{(0,\theta)}(\sigma')]\\
        &+ r_{i,j}[\mathbbm{1}(j=i+1)M_{n-j}^{(0,\theta)}(\sigma')+\mathbbm{1}(j=i-1)M_{n-j}^{(0,\theta)}(\sigma')]\\
        &+r_{1,0}\mathbbm{1}(i=1)M_n^{(0,\theta)}(\sigma',j)\\
        &=M_{n-j}^{(0,\theta)}(\sigma')[r_{i,i}^{(1)}\mathbbm{1}(j=i)+r_{i,i}^{(2)}\mathbbm{1}(j=i)+r_{i,j}(\mathbbm{1}(j=i+1)+\mathbbm{1}(j=i-1))]\\
        &+ r_{1,0}\mathbbm{1}(i=1)M_n^{(0,\theta)}(\sigma',j) 
    \end{align*}
    From Equation (1) we have that $M_n^{(0,\theta)}(\sigma,i)=M_{n-i}^{(0,\theta)}(\sigma)\upsilon_n^{(0,\theta)}(i).$ So,
    \begin{align*}
        \Lambda_nT_n^{(0,\theta)}(i,(\sigma',j))&=M_{n-j}^{(0,\theta)}(\sigma')[r_{i,i}^{(1)}\mathbbm{1}(j=i)+r_{i,i}^{(2)}\mathbbm{1}(j=i)+r_{i,j}(\mathbbm{1}(j=i+1)+\mathbbm{1}(j=i-1))]\\
        &+r_{1,0}\mathbbm{1}(i=1)M_{n-j}^{(0,\theta)}(\sigma')\upsilon_n^{(0,\theta)}(j)\\
        &=r_{i,j}M_{n-j}^{(0,\theta)}(\sigma')+r_{1,0}\mathbbm{1}(i=1)M_{n-j}^{(0,\theta)}(\sigma')\upsilon_n^{(0,\theta)}(j)\\
        &=(r_{i,j}+r_{1,0}\mathbbm{1}(i=1)\upsilon_n^{(0,\theta)}(j))M_{n-j}^{(0,\theta)}(\sigma')\\
        &=C_n(i,j)\Lambda_n(j,(\sigma',j))\\
        &=(C_n\Lambda_n)(i,(\sigma',j))
    \end{align*}
Note that $\Lambda(j,\cdot)$ is supported on \{ $\sigma \in C_n: \sigma_{l(\sigma)}=j$\}, and $\Lambda_n\psi_n$ is the identity kernel on $[n].$ Thus $Q_n^{(0,\theta)}=\Lambda_nT_n^{(0,\theta)}\psi_n=C_n\Lambda_n\psi_n=C_n,$ which gives us our explicit description of $Q_n^{(0,\theta)}.$ 
\end{proof}

\section{The K Oldest Alleles Process}

In a similar manner to the study of the rightmost column process of our Markov chain $(X_n^{(0,\theta)}(t))_{t\geq0}$ in $C_n$ we can study the K-th rightmost column process. For $l(\sigma)=l$ we have that $l>k$ or $l\leq k.$ However, for any $\sigma$ such that $l\leq k$ we can consider $\rho_\sigma=((0)_1^{k-l+1},\sigma).$ Clearly, $\rho_\sigma$ is an equivalent representation of the tuple $\sigma$ such that $l(\rho_\sigma)=k+1.$ Define a function for a tuple $\sigma$ in $C_n$ such that

\[
\psi_k(\sigma) =
\begin{cases}
(\sigma_{l-k+1},...,\sigma_l) & \text{if } l > k \\
((\rho_\sigma)_2^{k-l+1},\sigma) & \text{if } l \leq k 
\end{cases}
\]

This function gives rise to the K-th rightmost column process defined by $Y_n^{(0,\theta)}=\psi_k(X_n^{(0,\theta)})$ for $(X_n^{(0,\theta)}(t))_{t\geq0}.$ Let $\upsilon_{n,k}^{(0,\theta)}=M_n^{(0,\theta)}\circ \psi_k^{-1},$ the distribution of the K-th rightmost columns when the up-down chain is in stationarity. 

\begin{equation} \label{eq:2.1}
M_n^{(0,\theta)}(\sigma) =
\begin{cases}
M_{n-|\sigma_{l-k+1}^l|}^{(0,\theta)}(\sigma_1^{l-k})\upsilon_{n,k}^{(0,\theta)}(\sigma_{l-k+1}^l) & \text{if } l > k \\
M_{n-|\sigma|}^{(0,\theta)}((\rho_\sigma)_1)\upsilon_{n,k}^{(0,\theta)}((\rho_\sigma)_2^{k+1}) & \text{if } l \leq k 
\end{cases}
\end{equation}

Consider taking a $(0,\theta)$ up-step from $\sigma$ followed by a down-step. Let $E$ be the event in which this up-step stacks a box in any of the k-th rightmost columns of $\sigma$, and let $F$ be the event in which the down-step removes a box from the last k-th rightmost columns of a composition. Letting $w$ denote the total number of boxes in the k rightmost columns, we define $s_{w,w+1}=\mathbb{P}(E\cap F^c)$, $s_{w,w-1}=\mathbb{P}(E^c\cap F)$, $s_{w,w}^{(1)}=\mathbb{P}(E^c\cap F^c)$, and $s_{w,w}^{(2)}=\mathbb{P}(E\cap F)$. 

\begin{prop}
For $n\geq1,$ $\sigma=(\rho,\tau),\sigma'=(\rho',\tau')\in C_n$ such that $\tau$ and $\tau'$ represent the k rightmost columns, we have the following identity:

\begin{align*} 
T_n^{(0,\theta)} ((\rho,\tau),(\rho',\tau'))&=s_{w,w}^{(2)}\mathbbm{1}(\rho=(\rho',\tau'_1))T_{|\tau|}^{(0,0)}(\tau,(\tau')_2^k)\\
    &+s_{w,w}^{(2)}\mathbbm{1}(\rho=\rho')T_{|\tau|}^{(0,0)}(\tau,\tau')\\
    &+s_{w,w-1}p_{(0,\theta)}^{\uparrow}(\rho,(\rho',\tau'_1))p^{\downarrow}(\tau,(\tau')_2^k)\\
    &+s_{w,w-1}p_{(0,\theta)}^{\uparrow}(\rho,\rho')p^{\downarrow}(\tau,\tau')\\
    &+s_{w,w}^{(1)}\mathbbm{1}(\tau'=\tau)T^{(0,\theta)}_{|\rho|}(\rho,\rho')\\
    &+s_{w,w+1}p_{(0,0)}^{\uparrow}(\tau,\tau')p^{\downarrow}(\rho,\rho').
\end{align*}
\end{prop}

Let $1\leq k \leq n$. We define the set $B_{n,k}=\{\eta \in \bigcup_{i \in [n]}C_i | l(\eta)= k\}$ and the transition kernel $\Lambda_{n,k}$ from $B_{n,k}$ to $C_n$ by
$$\Lambda_{n,k}(\eta,(\rho,\eta))=M_{n-|\eta|}^{(0,\theta)}(\rho).$$
Let $\psi_{n,k}(\sigma,\gamma)=\mathbbm{1}(\sigma_{l(\sigma)-k}^{l(\sigma)}=\gamma)$ be a transition kernel from $C_n$ to $B_{n,k}$. 

\begin{prop}\label{propkoldest}
    For $n\geq1$, the transition kernel $Q_{n,k}^{(0,\theta)}=\Lambda_{n,k}T_n^{(0,\theta)}\psi_{n,k}$ satisfies 
    $$\Lambda_{n,k}T_n^{(0,\theta)}=Q_{n,k}^{(0,\theta)}\Lambda_{n,k}.$$
    Moreover, $Y_n^{(0,\theta)}$ is a time-homogeneous Markov chain with transition kernel $Q_{n,k}^{(0,\theta)}$ given explicitly by 
    \begin{align*}
        Q_{n,k}^{(0,\theta)}(\eta,\gamma)&=s_{w,w}^{(2)}\upsilon_{n-|\gamma|+\gamma_1}^{(0,\theta)}(\gamma_1)T_{|\eta|}^{(0,0)}(\eta, \gamma_2^k)\\
        &+s_{w,w}^{(2)}T_{|\eta|}^{(0,0)}(\eta, \gamma)\\
        &+s_{w,w-1}\upsilon_{n-|\gamma|+\gamma_1}^{(0,\theta)}(\gamma_1)p^{\downarrow}(\eta, \gamma_2^k)\\
        &+s_{w,w-1}p^{\downarrow}(\eta, \gamma)\\
        &+s_{w,w}^{(1)}\mathbbm{1}(\eta=\gamma)\\
        &+s_{w,w+1}p_{(0,0)}^{\uparrow}(\eta, \gamma)
    \end{align*}
    where $w=|\eta|$, when the initial distribution of $X_n^{(0,\theta)}$ is of the form $\mu\Lambda_{n,k}$.
\end{prop}
\begin{proof}
    Fix $\eta,\gamma \in B_{n,k}$ and $\rho' \in C_{n-|\gamma|}.$ Let $w=|\eta|$. We proceed by computing
\begin{align*}
\Lambda_{n,k}T_n^{(0,\theta)}(\eta,(\rho',\gamma))&=\sum\limits_{\rho \in C_{n-|\eta|}}\Lambda_{n,k}(\eta,(\rho,\eta))T_n^{(0,\theta)}((\rho,\eta),(\rho',\gamma))\\
&=\sum\limits_{\rho \in C_{n-|\eta|}}M_{n-|\eta|}^{(0,\theta)}(\rho)T_n^{(0,\theta)}((\rho,\eta),(\rho',\gamma))
\end{align*}
Using the identity from the previous proposition we expand this into the following sum:
\begin{align*}
&=s_{w,w}^{(2)}\sum\limits_{\rho \in C_{n-|\eta|}}M_{n-|\eta|}^{(0,\theta)}(\rho)\mathbbm{1}(\rho=(\rho',\gamma_1))T_{|\eta|}^{(0,0)}(\eta,\gamma_2^k)\\
&+s_{w,w}^{(2)}\sum\limits_{\rho \in C_{n-|\eta|}}M_{n-|\eta|}^{(0,\theta)}(\rho)\mathbbm{1}(\rho=\rho')T_{|\eta|}^{(0,0)}(\eta,\gamma)\\
&+s_{w,w-1}\sum\limits_{\rho \in C_{n-|\eta|}}M_{n-|\eta|}^{(0,\theta)}(\rho)p_{(0,\theta)}^{\uparrow}(\rho,(\rho',\gamma_1))p^{\downarrow}(\eta,\gamma_2^k)\\
&+s_{w,w-1}\sum\limits_{\rho \in C_{n-|\eta|}}M_{n-|\eta|}^{(0,\theta)}(\rho)p_{(0,\theta)}^{\uparrow}(\rho,\rho')p^{\downarrow}(\eta,\gamma)\\
&+s_{w,w}^{(1)}\sum\limits_{\rho \in C_{n-|\eta|}}M_{n-|\eta|}^{(0,\theta)}(\rho)\mathbbm{1}(\eta=\gamma)T^{(0,\theta)}_{|\rho|}(\rho,\rho')\\
&+s_{w,w+1}\sum\limits_{\rho \in C_{n-|\eta|}}M_{n-|\eta|}^{(0,\theta)}(\rho)p_{(0,0)}^{\uparrow}(\eta,\gamma)p^{\downarrow}(\rho,\rho')
\end{align*}

Using the properties $M_{N}^{(0,\theta)}p_{(0,\theta)}^{\uparrow}=M_{N+1}^{(0,\theta)}$ and $M_{|\rho'|+\gamma_1}^{(0,\theta)}(\rho',\gamma_1) = M_{|\rho'|}^{(0,\theta)}(\rho')\upsilon_{|\rho'|+\gamma_1}^{(0,\theta)}(\gamma_1)$, we can evaluate the sums and factor out $M_{n-|\gamma|}^{(0,\theta)}(\rho')$:
\begin{align*}
    \Lambda_{n,k}T_n^{(0,\theta)}(\eta,(\rho',\gamma)) &= M_{n-|\gamma|}^{(0,\theta)}(\rho') \bigg[ s_{w,w}^{(2)}\upsilon_{n-|\gamma|+\gamma_1}^{(0,\theta)}(\gamma_1)T_{|\eta|}^{(0,0)}(\eta,\gamma_2^k) + s_{w,w}^{(2)}T_{|\eta|}^{(0,0)}(\eta,\gamma)\\
    &+ s_{w,w-1}\upsilon_{n-|\gamma|+\gamma_1}^{(0,\theta)}(\gamma_1)p^{\downarrow}(\eta,\gamma_2^k) + s_{w,w-1}p^{\downarrow}(\eta,\gamma)\\
    &+ s_{w,w}^{(1)}\mathbbm{1}(\eta=\gamma) + s_{w,w+1}p_{(0,0)}^{\uparrow}(\eta,\gamma) \bigg]
\end{align*}

Factoring out the definition of $Q_{n,k}^{(0,\theta)}(\eta,\gamma)$ from the bracketed expression yields:
\begin{align*}
    \Lambda_{n,k}T_n^{(0,\theta)}(\eta,(\rho',\gamma))&=Q_{n,k}^{(0,\theta)}(\eta,\gamma)M_{n-|\gamma|}^{(0,\theta)}(\rho')\\
    &=Q_{n,k}^{(0,\theta)}(\eta,\gamma)\Lambda_{n,k}(\gamma,(\rho',\gamma))\\
    &=Q_{n,k}^{(0,\theta)}\Lambda_{n,k}(\eta,(\rho',\gamma)).
\end{align*}
\end{proof}

\section{The Diffusive Limit and Generator}

Define $\alpha_n = (n+\theta)(n+1)$ and the step-process $Z^n(t)$, as:
\begin{equation} \label{eq:process_scaling}
Z^n(t) = \frac{1}{n} Z_n^{(0,\theta)}(\lfloor \alpha_n t \rfloor), \quad t \ge 0.
\end{equation} 

In this section we show that $(Z^n(t))_{t\geq 0}$ converges to a diffusion on $[0,1]$ that jumps in from the boundary at $0$ and reflects from the boundary at $1$ when the boundary at $1$ is accessible.  In order to do this, we need to isolate the jumping in behavior of $Z_n^{(0,\theta)}$ from $1$.  This is somewhat complicated by the fact that there are two types transitions of $Z_n^{(0,\theta)}$ at $1$.  The first type, which is most common, comes from when the oldest allele does not go extinct and the second, rarer, type comes from when it does go extinct, resulting in a jump of $Z_n^{(0,\theta)}$ to the number of individuals with what had been the second oldest allele.

Specifically, we introduce the chain $\tilde Z_n^{(0,\theta)}$ on $\{0,1,\dots, n\}$ with transition matrix $\tilde Q_n^{(0,\theta)}$ defined by 
\begin{equation}\label{eqtilde} \tilde Q_n^{(0,\theta)}(i,j) = r_{i,j} \ \ \textrm{for} \ \  i\geq 1, j\geq 0\quad \textrm{and}\quad \tilde Q_n^{(0,\theta)}(0,j) = \upsilon_n^{(0,\theta)}(j).\end{equation}

There pathwise relationship between $\tilde Z_n^{(0,\theta)}$ and $Z_n^{(0,\theta)}$, in that $Z_n^{(0,\theta)}$ is $\tilde Z_n^{(0,\theta)}$ watched on $\{1,\dots, n\}$.  This is formalized in the text theorem.

\begin{thm}\label{thm:censored}
Let $Z_n^{(0,\theta)}$ and $\tilde Z_n^{(0,\theta)}$ start from $i\in \{1,\dots, n\}$.  Define $\beta^n_0=0$ and 
\[\beta^n_i = \inf\{k\geq \beta^n_{i-1} : \tilde Z^{(0,\theta)}(k) \in \{1,\dots,n \}\}\]
for $i\geq 1$.  Then $Z_n^{(0,\theta)}=_d (\tilde Z_n^{(0,\theta)}(\beta^n_k))_{k\geq 0}$.
\end{thm}

\begin{proof}
This is a standard construction, see e.g.\! \cite[Example 1.4.4]{Norris_1997} to see that $(\tilde Z_n^{(0,\theta)}(\beta^n_k))_{k\geq 0}$ is a Markov chain and the identification of its transition matrix with that of $Z_n^{(0,\theta)}$ follows from the definition of $\tilde Q_n^{(0,\theta)}$ and Proposition \ref{prop:rightmost}.
\end{proof}

We will first establish a scaling limit for
\begin{equation} \label{eq:cprocess_scaling}
\tilde Z^n(t) = \frac{1}{n} \tilde Z_n^{(0,\theta)}(\lfloor \alpha_n t \rfloor), \quad t \ge 0.
\end{equation}
and then use this coupling to transfer the limit to $Z^n$ by controlling the sequence $\beta^n_i$.  We define

Our first step is to derive the limiting jumping-in distribution from the boundary.  Next, we establish tight analytical bounds on the hitting times at the absorbing boundary using our modified chain.  Finally we concatenate these sample paths to prove global continuity and identify the limiting generator.

\subsection{Jumping-In Distribution}

\begin{prop} \label{prop:jump_in_dist}
If $\tilde Z^n(0)=0$ then $\tilde Z^n(1/\alpha_n) \Rightarrow V$ where the density of $V$ on $(0,1)$ is given by 
\begin{equation} \label{eq:jump_in_limit}
p(x) = \theta(1 - x)^{\theta - 1}.
\end{equation}
\end{prop}

\begin{proof}
Consider the inclusion mapping $\iota_n:\{0, 1, \dots, n\} \hookrightarrow [0, 1]$ by $\iota_n(k) = \frac{k}{n}$. This maps our discrete state space into a uniform lattice on $[0,1]$ where each state has a spatial cell volume of $v_n = \frac{1}{n}$. 

For any target point $x \in (0,1)$, we consider a sequence of discrete states $k_n$ such that their image under the inclusion map converges: $x_n = \iota_n(k_n) \to x$ as $n \to \infty$. 

From \cite[Theorem 3.3]{billingsley2013convergence}, it is sufficient to show that
\[ n \mathbb{P}_0(\tilde Z^n(1/\alpha_n)=x_n) = n \tilde Q_n^{(0,\theta)}(0, k_n)=n\nu_n^{(0,\theta)}(k_n) \to \theta(1 - x)^{\theta - 1}.\]
From \cite{GPRegen}, we know that 
\[
\nu_n^{(0,\theta)}(k_n)=  \binom{n}{k_n} \frac{k_n! \, \theta}{n(n - k_n + \theta)_{k_n}}.
\]
Consequently, expanding the binomial coefficient $\binom{n}{k_n} = \frac{(n)_{k_n}}{k_n!}$ and canceling out some terms,
\[
n\nu_n^{(0,\theta)}(k_n) = \theta \left[ n \cdot \frac{(n)_{k_n}}{k_n!} \right] \frac{k_n!}{n(n - k_n + \theta)_{k_n}} = \theta \frac{(n)_{k_n}}{(n - k_n + \theta)_{k_n}}.
\]
Using the ratio of Gamma functions to represent the falling and rising factorials:
\[
\frac{(n)_{k_n}}{(n - k_n + \theta)_{k_n}} = \frac{\Gamma(n+1)}{\Gamma(n-k_n+1)} \frac{\Gamma(n-k_n+\theta)}{\Gamma(n+\theta)}.
\]
By the standard asymptotic property $\frac{\Gamma(z+a)}{\Gamma(z+b)} \sim z^{a-b}$ as $z \to \infty$, and recalling that our inclusion map implies $n-k_n = n(1-x_n)$, we observe that as $n \to \infty$:
\[
\frac{\Gamma(n+1)}{\Gamma(n+\theta)} \sim n^{1-\theta} \quad \text{and} \quad \frac{\Gamma(n-k_n+\theta)}{\Gamma(n-k_n+1)} \sim (n-k_n)^{\theta-1}.
\]
Combining these limits, we obtain our final density:
\[\begin{split}
n \mathbb{P}_0(\tilde Z^n(1/\alpha_n)=x_n)  = n\nu_n^{(0,\theta)}(k_n) & \sim 1 \cdot \theta \cdot n^{1-\theta} \cdot (n(1-x_n))^{\theta-1}\\
& = \theta n^{1-\theta} n^{\theta-1} (1-x_n)^{\theta-1} \\
& = \theta (1-x_n)^{\theta-1}.\end{split}
\]
As $x_n \to x$, the density converges to $p(x) = \theta(1-x)^{\theta-1}$, which completes the proof.
\end{proof}

\subsection{Sample Path and Hitting Times Convergence}
Our approach to the diffusive limit first requires establishing the convergence of $\tilde Z^n$ stopped at the first time it hits zero jointly with the time at which it first hits zero.  For this, we follow the approach in \cite[Chapter 10]{ethier2009markov}.  

Let $\hat Z$ be a diffusion on $[0,1]$ with generator  
\[ \hat Gf(x) = x(1-x)f''(x) - \theta x f'(x) \]
with the set of polynomials being a core of the domain of $\hat G$.  See \cite[Theorem 8.2.8]{ethier2009markov} for the proof that $\hat Z$ is well defined.  As noted in \cite{ethier2009markov}, the generator can also be though of as the closure of $\{(f,Gf): f\in C^2([0,1])\}$.  Let $\tilde \tau_n = \inf\{ t: \tilde Z^n(t)=0\}$ and $\tau = \inf\{ t: \hat Z(t)=0\}$.

\begin{thm} \label{thm:hitting_time_convergence}
If $\tilde{Z}^n(0) \Rightarrow \hat Z(0)$, then $(\tilde{Z}^n, \tilde{\tau}_n) \Rightarrow (\hat Z, \tau)$ in $D_{[0,1]}[0, \infty) \times [0, \infty]$.
\end{thm}

In order to prove this, we first show that $\tilde{Z}^n\Rightarrow \hat Z$ and then address the joint convergence with the hitting time of $0$.

\begin{thm} 
If $Z^n(0) \Rightarrow Z(0)$ in $[0,1]$, then $\tilde Z^n (\cdot \wedge \tilde \tau_n) \Rightarrow \hat Z$ in the Skorohod space $D_{[0,1]}[0,\infty)$.  
\end{thm}

\begin{proof}
To establish the weak convergence $Z^n \Rightarrow Z$, we rely on the convergence of the strongly continuous semigroups associated with the discrete processes to the semigroup of the limiting diffusion. Let $K_n=\{0,1/n,2/n,\dots,1\}$.  By \cite[Theorem 1.6.5]{ethier2009markov}, and following the framework of Chapter 10, Theorem 1.1, it suffices to show that the discrete generators $\tilde G_n$ defined by
\[ \tilde G_n f(x) = \alpha_n \mathbb{E}_x\left[f(\tilde Z^n(\alpha_n^{-1}\wedge \tilde \tau_n) - f(x)\right]\]
converge uniformly to the continuous generator $\hat G$ for all functions in a core of $\hat G$ in the sense that
\[ \lim_{n\to\infty} \sup_{x\in K_n} |\tilde G_n f(x) - \hat G f(x)| =0.\]
As noted above, polynomials form a core for $G$, but our argument applies equally to $f\in C^3([0,1])$.  Fix $f\in C^3([0,1])$.  Note also that $\tilde G_n f(0)=\hat G f(0)=0$, so we need only consider $x\neq 0$.  Let $\Delta \tilde Z^n = \tilde Z^n(\alpha_n^{-1})-\tilde Z^n(0)$.  Applying a third-order Taylor expansion to $f$ around $x\in K_n$, we have:
\begin{equation}\label{eqconv1} \mathbb{E}_x\left[f(\tilde Z^n(\alpha_n^{-1}\wedge \tilde \tau_n)- f(x)\right] = f'(x)\E_x (\Delta \tilde Z^n) + \frac{1}{2}f''(x)\E_x\left[(\Delta \tilde Z^n)^2\right] + \E_x R_f(x, \Delta \tilde Z^n), \end{equation}
where the remainder term satisfies $|R_f(x,\delta )| \le \frac{1}{6}\|f'''\|_\infty |\delta |^3$. 

Using the transition probabilities from Equation \eqref{eqtilde}, we calculate the first moment for a state $x \in K_n\setminus \{0\}$:
\begin{align}\label{eqm1}
\mathbb{E}_x[\Delta \tilde Z^n] &= \frac{1}{n}\left(\frac{(n-xn)xn}{(n+\theta)(n+1)}\right) - \frac{1}{n}\left(\frac{(n+\theta-xn)xn}{(n+\theta)(n+1)}\right) \\
&= \frac{-x\theta}{(n+\theta)(n+1)} = \frac{-x\theta}{\alpha_n}.
\end{align}
Similarly, we calculate the second moment:
\begin{align}\label{eqm2}
\mathbb{E}_x[(\Delta \tilde Z^n)^2] &= \frac{1}{n^2}\left(\frac{xn(2n-2xn+\theta)}{(n+\theta)(n+1)}\right) = \frac{x(2(1-x) + \theta/n)}{\alpha_n}.
\end{align}
Plugging these into Equation \eqref{eqconv1} we see that
\begin{align*} 
\tilde G_n f(x) &= \alpha_n \mathbb{E}_x\left[f(\tilde Z^n(\alpha_n^{-1}\wedge \tilde \tau_n)- f(x)\right] \\
&= f'(x)(-\theta x) + \frac{1}{2}f''(x)\left[2x(1-x) + \frac{\theta x}{n}\right] + \alpha_n \mathbb{E}_x[R_f(x, \Delta \tilde Z^n)] \\
&= \left[ x(1-x)f''(x) - \theta x f'(x) \right] + \frac{\theta x}{2n}f''(x) + \alpha_n \mathbb{E}_x[R_f(x, \Delta \tilde Z^n)] \\
&= \hat Gf(x) + \frac{\theta x}{2n}f''(x) + \alpha_n \mathbb{E}_x[R_f(x, \Delta \tilde Z^n)].
\end{align*}

It remains to show that the second two terms on the right-hand-side vanish uniformly over $x \in K_n$ as $n \to \infty$. For the second derivative error term, we have:
\[ \sup_{x \in K_n} \left| \frac{\theta x}{2n} f''(x) \right| \le \frac{\theta}{2n} \|f''\|_\infty, \]
which converges uniformly to 0. 

For the Taylor remainder, recall that the process only steps to adjacent states, meaning the maximum jump size is bounded by $|\Delta \tilde Z^n| \le 1/n$. Consequently, $|\Delta \tilde Z^n|^3 \le \frac{1}{n^3}$. As a result,  
\begin{align*} 
\sup_{x \in K_n} |\alpha_n \mathbb{E}_x[R_f(x, \Delta \tilde Z^n)]| &\le \frac{1}{6}\|f'''\|_\infty \sup_{x \in K_n} \alpha_n \mathbb{E}_x[|\Delta \tilde Z^n|^3] \\
&\le \frac{\alpha_n}{6n^3}\|f'''\|_\infty  \\
&= \frac{(n+1)(n+\theta)}{6n^3}\|f'''\|_\infty\\
& \to 0,
\end{align*}
with the convergence being uniform on $[0,1]$.
\end{proof}

We now by characterize the expected hitting time $0$ of the limiting diffusion.  Let 
\[ \tau = \inf\{t \ge 0 : \hat Z(t) = 0\}, \] 
denote the first hitting time $0$ of $\hat Z$.

\begin{prop} \label{prop:expected_hitting_time}
Let $g_0$ be the unique $C([0,1]) \cap C^2((0,1])$ solution of the differential equation $\hat Gg_0 = -1$ with boundary conditions $g_0(0) = 0$, and $g_0'(1)$ finite. Then $\mathbb{E}_x[\tau] = g_0(x)$ for all $x \in [0,1]$. Consequently, the expected hitting time is explicitly given by:
\[ \mathbb{E}_x[\tau] = \int_0^x (1-u)^{-\theta} \left[ \int_u^1 \frac{(1-y)^{\theta-1}}{y} \,dy \right] du. \]
\end{prop}

\begin{proof}
This is the same as \cite[Proposition 10.2.8]{ethier2009markov} under the change of coordinates $x\mapsto 1-x$.
\end{proof}

\begin{corollary} \label{cor:g0_bounds}
Let $g_0$ be defined as in Proposition \ref{prop:expected_hitting_time} there exist positive constants $c_1, c_2, c_3$ such that for all $x \in (0, 1)$:
\begin{align}
    0 < g_0'(x) &\le c_1(1 - \ln x), \label{eq:g0_bound1} \\
    |g_0''(x)| &\le c_2 x^{-1}, \label{eq:g0_bound2} \\
    |g_0'''(x)| &\le c_3 x^{-2}. \label{eq:g0_bound3}
\end{align}
Furthermore, there exists a neighborhood $(0, \delta)$ near the absorbing boundary $0$ such that $g_0''(x) < 0$, making $g_0$ strictly concave there.
\end{corollary}

\begin{proof}
The first derivative is given by:
\[
g_0'(x) = (1-x)^{-\theta} \int_x^1 y^{-1}(1-y)^{\theta-1} dy = \frac{\int_x^1 y^{-1}(1-y)^{\theta-1} dy}{(1-x)^\theta}.
\]
To establish the strict positivity $0 < g_0'(x)$, observe that for $x \in (0, 1)$ and $\theta > 0$, the integrand $y^{-1}(1-y)^{\theta-1}$ is strictly positive on the integration interval $(x, 1)$. Multiplying this positive integral by the strictly positive term $(1-x)^{-\theta}$ guarantees that $g_0'(x) > 0$.

Note that near the absorbing boundary, for $x \in (0, 1/2]$, the term $(1-x)^{-\theta}$ is bounded by $2^\theta$. The integral can be bounded by a constant multiple of $\int_x^{1/2} y^{-1} dy = \ln(1/2) - \ln(x)$. 

At the right boundary, we evaluate the limit of $g_0'(x)$ as $x \to 1$. Since $\theta > 0$, the denominator $(1-x)^\theta$ approaches $0$. Furthermore, because the integrand is integrable for $\theta > 0$, the numerator integral $\int_x^1 y^{-1}(1-y)^{\theta-1} dy$ also approaches $0$ as the interval of integration shrinks to a point. Because this presents a valid $0/0$ indeterminate form and both functions are differentiable on $(0,1)$, we can apply L'Hôpital's rule:
\[
\lim_{x \to 1} g_0'(x) = \lim_{x \to 1} \frac{\int_x^1 y^{-1}(1-y)^{\theta-1} dy}{(1-x)^\theta} = \lim_{x \to 1} \frac{-x^{-1}(1-x)^{\theta-1}}{-\theta(1-x)^{\theta-1}} = \frac{1}{\theta}.
\]
Since the limit is finite and $g_0'$ is continuous, $g_0'$ is bounded by its finite maximum on $[1/2, 1)$. Combining this with the logarithmic behavior near zero, there exists a global constant $c_1 > 0$ such that for all $x \in (0, 1)$:
\[
0 < g_0'(x) \le c_1(1 - \ln x).
\]
By the definition of our operator, we have the relation:
\[
g_0''(x) = \frac{\theta x g_0'(x) - 1}{x(1-x)}.
\]
As $x \to 1$, the numerator approaches $\theta(1)(1/\theta) - 1 = 0$, and the denominator approaches $0$. Applying L'Hôpital's rule to this $0/0$ form, we find $\lim_{x \to 1} g_0''(x) = -1/(\theta+1)$, which is also finite. Near zero, substituting our global logarithmic bound into the numerator yields:
\[
|g_0''(x)| \le \frac{|\theta x c_1(1 - \ln x) - 1|}{x(1-x)}.
\]
Since $x(1 - \ln x) \to 0$ as $x \to 0$, the numerator is globally bounded. The denominator introduces an $x^{-1}$ singularity at zero. Because there is no singularity at $x=1$, there exists a constant $c_2 > 0$ such that for all $x \in (0, 1)$:
\[
|g_0''(x)| \le c_2 x^{-1}.
\]
The strict concavity near the boundary follows because $\lim_{x \to 0^+} x g_0'(x) = 0$, ensuring $\theta x g_0'(x) - 1 < 0$ for sufficiently small $x$.

Finally, differentiating the relation for $g_0''(x)$ yields:
\[
g_0'''(x) = \frac{\theta g_0''(x)}{1-x} + \frac{\theta g_0'(x)}{(1-x)^2} + \frac{1-2x}{x^2(1-x)^2}.
\]
By a similar argument, applying the triangle inequality, and substituting our established global bounds for $g_0'$ and $g_0''$:
\[
|g_0'''(x)| \le \frac{\theta c_2 x^{-1}}{1-x} + \frac{\theta c_1(1-\ln x)}{(1-x)^2} + \frac{|1-2x|}{x^2(1-x)^2}.
\]
Factoring out $x^{-2}$ from the right side explicitly reveals the growth rate:
\[
|g_0'''(x)| \le x^{-2} \left( \frac{\theta c_2 x}{1-x} + \frac{\theta c_1 x^2(1-\ln x)}{(1-x)^2} + \frac{|1-2x|}{(1-x)^2} \right).
\]
For $x \in (0, 1/2]$, the denominator $(1-x)$ is strictly bounded away from zero (specifically, $1-x \ge 1/2$). Because $x$ and $x^2(1-\ln x)$ are continuous functions that approach zero as $x \to 0$, the entire expression inside the parentheses is strictly bounded above by some finite constant $K$. Combining this explicit bounding near zero with the finite limit at the right boundary, there exists a global constant $c_3 > 0$ such that for all $x \in (0, 1)$:
\[
|g_0'''(x)| \le c_3 x^{-2}.
\qedhere\]
\end{proof}

We now turn to $\tilde \tau_n = \inf\{t : \tilde Z^n(t)=0\}$.  Note that $\tilde \tau_n \Rightarrow \tau$ does not follow directly from $\tilde Z^n\Rightarrow \hat Z$ because the hitting time of zero is not continuous on the Skorokhod space. 

\begin{lemma} 
Let $g_0$ be the function defined in Proposition \ref{prop:expected_hitting_time}. There exist a positive constant $\kappa$ and an integer $n_0$ such that for all $n \ge n_0$ and all $x \in K_n$,
\[ \mathbb{E}_x[\tilde{\tau}_n] \le \kappa g_0(x). \]
\end{lemma}

\begin{proof}
As before, we analyze the operator
\[ \tilde G_n f(x) = \alpha_n \mathbb{E}_x\left[f(\tilde Z^n(\alpha_n^{-1}\wedge \tilde \tau_n) - f(x)\right].\]
Continuing with the notation $\Delta \tilde Z^n = \tilde Z^n(\alpha_n^{-1})-\tilde Z^n(0)$, we write
\[ \tilde{G}_n f(x) = \alpha_n \mathbb{E}_x[f(x + \Delta \tilde{Z^n}) - f(x)]. \]

For $0 < \xi < 1$, define $V(\xi) = [\xi, 1]$ and $V_n(\xi) = K_n \cap V(\xi)$. We first analyze the operator away from the absorbing boundary. Applying a third-order Taylor expansion to $\tilde{G}_n g_0(x)$ yields:
\[ \tilde{G}_n g_0(x) = \alpha_n \mathbb{E}_x\left[g_0'(x)\Delta \tilde Z^n + \frac{1}{2}g_0''(x)(\Delta \tilde Z^n)^2\right] + R_n(x), \]
where the remainder term is given by $R_n(x) = \frac{1}{6}\alpha_n \mathbb{E}_x[g_0'''(x^*)(\Delta \tilde Z^n)^3]$, for some (random) $x^*$ between $x$ and $x + \Delta \tilde Z^n$. Substituting our explicit moments from Equations \eqref{eqm1} and \eqref{eqm2} into the main terms, we obtain:
\begin{align*}
\tilde{G}_n g_0(x) &= -\theta x g_0'(x) + \frac{1}{2}g_0''(x)\left[2x(1-x) + \frac{\theta x}{n}\right] + R_n(x) \\
&= \left[x(1-x)g_0''(x) - \theta x g_0'(x)\right] + \frac{\theta x}{2n}g_0''(x) + R_n(x) \\
&= \hat Gg_0(x) + \frac{\theta x}{2n}g_0''(x) + R_n(x).
\end{align*}

We now examine the uniform convergence of $\tilde{G}_n g_0(x)$ to $Gg_0(x)$ by defining our domain away from the absorbing boundary as $V_n(m/n)$ for an integer $m \ge 2$. We strictly control the error terms using the global bounds established in Corollary \ref{cor:g0_bounds}. 
For the second derivative term, applying \eqref{eq:g0_bound2} yields $|\frac{\theta x}{2n}g_0''(x)| \le \frac{\theta c_2}{2n}$, which vanishes uniformly as $n \to \infty$. 

To bound the remainder $R_n(x)$, observe that since $\Delta \tilde Z^n \in \{-1/n, 0, 1/n\}$, we have $|\Delta \tilde Z^n|^3 \leq \frac{1}{n}(\Delta \tilde Z^n)^2$. The third absolute moment is therefore bounded by:
\[ \alpha_n \mathbb{E}_x[|\Delta \tilde Z^n|^3] = \frac{1}{n} \alpha_n \mathbb{E}_x[(\Delta \tilde Z^n)^2] = \frac{x(2(1-x) + \theta/n)}{n} \le \frac{Kx}{n}, \]
for some constant $K > 0$. For $x \in V_n(m/n)$, we have $x \ge m/n$. Since the jump size is $1/n$, the intermediate point $x^*$ satisfies $x^* \ge x - 1/n \ge x(1 - 1/m) = x(\frac{m-1}{m})$. Applying \eqref{eq:g0_bound3}, we get:
\[ |g_0'''(x^*)| \le c_3(x^*)^{-2} \le c_3 \left(\frac{m}{m-1}\right)^2 x^{-2}. \]
Substituting these into the remainder gives:
\[ |R_n(x)| \le \frac{1}{6} \left[ c_3 \left(\frac{m}{m-1}\right)^2 x^{-2} \right] \left( \frac{Kx}{n} \right) = \frac{C_m}{nx}, \]
where $C_m = \frac{K c_3}{6} (\frac{m}{m-1})^2$. Because $x \ge m/n$, it follows that $nx \ge m$, meaning $|R_n(x)| \le \frac{C_m}{m}$. We can now cleanly bound the total difference:
\[ \sup_{x \in V_n(m/n)} |\tilde{G}_n g_0(x) - \hat Gg_0(x)| \le \frac{\theta c_2}{2n} + \frac{C_m}{m}. \]
Recall that $Gg_0(x) = -1$. Taking the limit supremum as $n \to \infty$, the $1/n$ term vanishes, leaving an error bounded solely by $C_m/m$. Since $C_m$ converges to a constant as $m \to \infty$, taking the limit as $m \to \infty$ drives the remaining error to zero. Thus, we conclude:
\begin{equation} \label{eq:interior_limit}
\lim_{m \to \infty} \limsup_{n \to \infty} \sup_{x \in V_n(m/n)} \tilde{G}_n g_0(x) \le -1.
\end{equation}

Next, we address the region near the absorbing boundary. Let $x_n = m/n$ for small integer $m \ge 1$. From our first moment calculation, we can show that:
\[ \mathbb{E}_{x_n}[x_n + \Delta \tilde{Z}^n] = x_n - \frac{\theta x_n}{\alpha_n} < x_n. \]
From Corollary \ref{cor:g0_bounds}, we know that near zero, $g_0'(x) > 0$ and $g_0$ is strictly concave. Thus, for sufficiently large $n$, applying Jensen's inequality yields:
\[ \mathbb{E}_{x_n}[g_0(x_n + \Delta \tilde{Z}^n)] < g_0(\mathbb{E}_{x_n}[x_n + \Delta \tilde{Z}^n]) < g_0(x_n). \]
Consequently, $\tilde{G}_n g_0(x_n) = \alpha_n (\mathbb{E}_{x_n}[g_0(x_n + \Delta \tilde{Z}^n)] - g_0(x_n)) < 0$. Taking the limit superior as $n \to \infty$, we get:
\begin{equation} \label{eq:boundary_limit}
\limsup_{n \to \infty} \tilde{G}_n g_0(m/n) < 0, \quad m = 1, 2, \dots
\end{equation}

We now combine our two main bounds. The strict negativity near the boundary established in Equation \eqref{eq:boundary_limit}, together with the uniform $-1$ bound in the interior derived in Equation \eqref{eq:interior_limit}, guarantee that there exists a constant $C > 0$ and an integer $n_0$ such that for all $n \ge n_0$ and all $x \in K_n \setminus \{0\}$,
\[ \tilde{G}_n g_0(x) \le -C. \]
To conclude the proof, we first note that the following is a martingale:
\[ M_k = g_0(\tilde{Z}^n((k/\alpha_n)\wedge \tilde \tau_n)) - \sum_{j=0}^{k-1} \frac{1}{\alpha_n} \tilde{G}_n g_0(\tilde{Z}^n((j/\alpha_n)\wedge \tilde \tau_n)). \]
Indeed, since $\tilde{Z}^n((k/\alpha_n)\wedge \tilde \tau_n) = \tilde Z^{(0,\theta)}_n(k \wedge (\alpha_n \tilde \tau_n))$ a Markov chain, this is one of the standard martingales associated with Markov chains. By the Optional Sampling Theorem and our uniform bound $\tilde{G}_n g_0 \le -C$ on $K_n \setminus \{0\}$, we have:
\begin{align*}
\mathbb{E}_x[g_0(\tilde{Z}^n(t\wedge \tilde{\tau}_n))] &= g_0(x) + \mathbb{E}_x\left[\sum_{j=0}^{\alpha_n(\tilde{\tau}_n \wedge t) - 1} \frac{1}{\alpha_n} \tilde{G}_n g_0\left(\tilde{Z}^n\left(\frac{j}{\alpha_n}\right)\right)\right] \\
&\le g_0(x) - \mathbb{E}_x\left[\sum_{j=0}^{\alpha_n(\tilde{\tau}_n \wedge t) - 1} \frac{C}{\alpha_n}\right] \\
&= g_0(x) - C \cdot \mathbb{E}_x[t\wedge \tilde{\tau}_n], \quad \text{for } x \in K_n \setminus \{0\}, t=0, 1/\alpha_n, 2/\alpha_n, \dots
\end{align*}
Since $g_0$ is a non-negative function, $\mathbb{E}_x[g_0(\tilde{Z}^n(t\wedge \tilde{\tau}_n))] \ge 0$. Rearranging the inequality yields:
\[ C \cdot \mathbb{E}_x[t\wedge\tilde{\tau}_n] \le g_0(x). \]
We now let $t \to \infty$. By the Monotone Convergence Theorem, the expected value $\mathbb{E}_x[t\wedge\tilde{\tau}_n ]$ monotonically increases to $\mathbb{E}_x[\tilde{\tau}_n]$. Thus, we conclude:
\[ C \cdot \mathbb{E}_x[\tilde{\tau}_n] \le g_0(x). \]
Setting $\kappa = 1/C$ gives $\mathbb{E}_x[\tilde{\tau}_n] \le \kappa g_0(x)$.
\end{proof}

To establish the joint weak convergence in Theorem \ref{thm:hitting_time_convergence}, we first consider the hitting time of $\delta>0$, which is a.s. continuous for $\hat Z$. 

\begin{lemma} \label{lem:hitting_time_continuity}
For $\delta>0$, define the $\delta$-hitting time mapping $\tau^{\delta}: D_{[0,1]}[0,\infty) \to [0,\infty]$ by $\tau^{\delta}(x) = \inf\{t \ge 0 : x(t) \le \delta\}$. The mapping $x \mapsto \tau^{\delta}(x)$ is continuous at any path $x \in D_{[0,1]}[0,\infty)$ as long as $x$ is continuous at $\tau^{\delta}(x)$ and $\tau^{\delta}(x) = \tilde{\tau}^{\delta}(x)$, where $\tilde{\tau}^{\delta}(x) = \inf\{t \ge 0 : x(t) < \delta\}$. If $Z$ is the continuous diffusion process governed by $G$ on $[0,1]$, then the mapping $\tau^{\delta}$ is continuous almost surely with respect to the distribution of $Z$.
\end{lemma}

This result is quite classical, but proofs can be hard to locate, so we include one following the approach of \cite{rogers2022ray}.

\begin{proof}
Let $x \in D_{[0,1]}[0,\infty)$ be a path satisfying the two conditions from the hypothesis, and let $t = \tau^{\delta}(x)$. Assume $t < \infty$. Let $x_n \to x$ in the Skorohod topology; we must show $\tau^{\delta}(x_n) \to t$.

Let $\epsilon > 0$. Because $\tau^{\delta}(x) = \tilde{\tau}^{\delta}(x)$, the path strictly crosses the threshold. Then there exists some time $s \in [t, t + \epsilon)$ such that $x(s) < \delta$. By the definition of Skorohod convergence, there exist strictly increasing, continuous time changes $\lambda_n$ mapping $[0,\infty)$ onto itself such that $\lambda_n(u) \to u$ uniformly on compact sets, and $x_n(\lambda_n(u)) \to x(u)$ uniformly on compact sets. Evaluating at $s$, we have $x_n(\lambda_n(s)) \to x(s) < \delta$. For all sufficiently large $n$, $x_n(\lambda_n(s)) < \delta$, which implies $\tau^{\delta}(x_n) \le \lambda_n(s)$. Because $\lambda_n(s) \to s < t + \epsilon$, taking the limit superior yields $\limsup_{n \to \infty} \tau^{\delta}(x_n) \le t + \epsilon$.

To establish the lower bound, observe that by the definition of $\tau^{\delta}(x)$, $x(u) > \delta$ for all $u \in [0, t - \epsilon]$. Because $x$ has càdlàg paths, its infimum over the compact interval $[0, t - \epsilon]$ is achieved and is strictly greater than $\delta$. Let this infimum be $\delta' > \delta$. By the uniform convergence of $x_n \circ \lambda_n$ to $x$, for all sufficiently large $n$, we have $\inf_{u \in [0, t - \epsilon]} x_n(\lambda_n(u)) > \delta$. Thus, the process $x_n$ cannot drop below $\delta$ before time $\lambda_n(t - \epsilon)$. Since $\lambda_n(t-\epsilon) \to t-\epsilon$, it follows that $\liminf_{n \to \infty} \tau^{\delta}(x_n) \ge t - \epsilon$. Since $\epsilon$ is arbitrary, $\tau^{\delta}(x_n) \to t$, establishing the continuity of the mapping.

Now, for $Z$ our continuous diffusion limit we note that $Z$ has continuous sample paths almost surely. Moreover, $\delta \in (0,1)$ is a regular interior point for the diffusion $Z$ because the diffusion coefficient $a(\delta) = \delta(1-\delta) > 0$ is strictly non-degenerate, and the drift $b(\delta) = -\theta \delta$ is bounded. By the strong Markov property and the properties of regular one-dimensional diffusions, the process does not linger or reflect at an interior point without crossing. Then upon hitting $\delta$, the process immediately enters the open interval $(0, \delta)$ with probability 1. Therefore, the hitting time $\tau^{\delta}(Z)$ and the strict crossing time $\tilde{\tau}^{\delta}(Z)$ coincide almost surely. Hence, the set of sample paths for which the hitting time mapping is discontinuous has measure zero.
\end{proof}

\begin{proof}[Proof of Theorem \ref{thm:hitting_time_convergence}]
Consider $h(x) = (x, \tau(x))$ and $h_\delta(x) = (x, \tau^\delta(x))$ for $\delta > 0$. 
By Lemma \ref{lem:hitting_time_continuity}, the mapping $\tau^\delta$ is continuous almost surely with respect to the distribution of $Z$. Furthermore, as $\delta \to 0$, $\tau^\delta(x) \to \tau^0(x)$ almost surely, satisfying condition. Let $\tilde \tau^\delta_n = \tau^\delta(\tilde Z^n)$ and note that $\tilde \tau_n = \tau(\tilde Z^n)$.  We next verify that
\[ \lim_{\delta \to 0} \limsup_{n \to \infty} \P(\rho(\tilde{\tau}_n, \tilde\tau_n^\delta) > \varepsilon) = 0, \]
where $\rho(s, t) = |\arctan(s) - \arctan(t)|$ denotes the metric on $[0, \infty]$. By using the fact that the derivative of the arctangent function is strictly bounded by $1$, we have $|\arctan(s) - \arctan(t)| \le |s - t|$. This implies:
\[ \P(\rho(\tilde{\tau}_n, \tilde\tau_n^\delta) > \varepsilon) \le \P(|\tilde{\tau}_n - \tilde\tau_n^\delta| > \varepsilon). \]
Because $\tau_n^\delta$ is a well-defined stopping time, we apply the Strong Markov Property at time $\tau_n^\delta$ to restart the process. The shifted process $\tilde{Z}^n(t)$ starts at $x > \delta$ and has continuous trajectories in the Skorohod topology (with maximum jump size $1/n \to 0$), so it must cross the level $\delta$ before it can hit the origin. Therefore, $\tilde{\tau}_n \ge \tilde \tau_n^\delta$ almost surely, which resolves the absolute value. Applying this uniform bound, we have:
\[ \P(\rho(\tilde{\tau}_n, \tilde\tau_n^\delta) > \varepsilon) \le \frac{\kappa}{\varepsilon} \mathbb{E}\left[ g_0(\tilde{Z}^n(\tilde\tau_n^\delta)) \right]. \]
At the stopping time $\tilde\tau_n^\delta$, we have $\tilde{Z}^n(\tilde\tau_n^\delta) \le \delta$. From Proposition \ref{prop:expected_hitting_time}, $g_0(x)$ is continuous with $g_0(0) = 0$. Taking the limit as $\delta \to 0$, $g_0(\tilde{Z}^n(\tilde\tau_n^\delta)) \to 0$ uniformly bounding the probability to $0$. Thus, by \cite[Problem 3.5]{ethier2009markov}, $(\tilde{Z}^n, \tilde{\tau}_n) \Rightarrow (Z, \tau^0)$.
\end{proof}

\subsection{Sample Path Concatenation and Continuity}

We must first establish the conditions under which the concatenation of such sample paths preserves convergence in the Skorohod topology.

\begin{thm} \label{thm:concatenation}
    Let $f, f^1, f^2, \dots$ be continuous functions in the Skorohod space $D_E[0,\infty)$ taking values in a metric space $(E, d)$. Let $\{\tau^i\}_{i\geq 0}$ be a sequence of strictly positive times such that $\sum_{i=0}^\infty \tau^i = \infty$.  Define the following concatenation with $t \in [0,\infty)$:
    
\[
(f\star f^1\star f^2\star \dots)(t) =
\begin{cases}
f(t) & \text{if } \quad 0\leq t<\tau^0 \\
f^1(t-\tau^0) & \text{if } \quad \tau^0\leq t<\tau^0 + \tau^1 \\
\vdots \\
f^n\left(t-\sum \limits_{i=0}^{n-1}\tau ^i\right)  & \text{if } \quad\sum \limits_{i=0}^{n-1}\tau ^i \leq t < \sum \limits_{i=0}^{n}\tau ^i\\
\vdots
\end{cases}
\]

\noindent Suppose that there exists a collection of sequences of functions $(f_n), (f_n^1), (f_n^2), \dots$ with corresponding sequences of strictly positive times $(\sigma^0_n), (\sigma^1_n), \dots$ representing the concatenation times for the pre-limit paths. Suppose that as $n \to \infty$, we have convergence in the product space $D_E[0,\infty) \times (0,\infty)$ given by:
\[ (f_n, \sigma_n^0) \longrightarrow (f, \tau^0) \quad \text{and} \quad (f_n^i, \sigma_n^i) \longrightarrow (f^i, \tau^i) \quad \text{for all } i \ge 1. \]

If the concatenated pre-limit path is defined as:
\[
(f_n\star f^1_n\star f^2_n\star \dots)(t) =
\begin{cases}
f_n(t) & \text{if } \quad 0\leq t<\sigma_n^0 \\
f_n^1(t-\sigma_n^0) & \text{if } \quad \sigma_n^0\leq t<\sigma_n^0 + \sigma_n^1 \\
\vdots \\
f_n^k\left(t-\sum \limits_{i=0}^{k-1}\sigma_n^i\right)  & \text{if } \quad\sum \limits_{i=0}^{k-1}\sigma_n^i \leq t < \sum \limits_{i=0}^{k}\sigma_n^i\\
\vdots
\end{cases}
\]
and if $f, f^1, f^2, \dots$ are continuous at the endpoints of their respective concatenated intervals, then there is convergence $(f_n\star f^1_n\star f^2_n\star \dots)\longrightarrow (f\star f^1\star f^2\star \dots)$ in the Skorokhod topology.
\end{thm}

\begin{proof}
    It is sufficient to prove that for an arbitrary continuity point $t$ of\\
    $(f\star f^1\star f^2\star \dots)$ there exist some elements $\lambda_n \in \Lambda_t,$ the space of continuous and increasing mappings in $[0,t]$, such that 
    \[
\sup_{s \leq t} |\lambda_ns -s|\longrightarrow 0 \ \text{and},
\]
\[
\sup_{s \leq t} d((f_n\star f^1_n\star f^2_n\star \dots)(\lambda_ns), (f\star f^1\star f^2\star \dots)(s))\longrightarrow 0.
\] 

It is important to note that the only points that are not continuity points of $(f\star f^1\star f^2\star \dots)$ are the times at the jumps of our piecewise function. Hence if $t$ is a continuity point, then $t \notin \{\tau^0, \tau^0 + \tau^1, \dots, \sum_{i=0}^{k}\tau ^i, \dots\}.$ Suppose that $t> \tau^0$ and let $k\geq 0$ be such that $t<\sum \limits_{i=0}^{k}\tau ^i$. Without loss of generality assume that $k$ is the first instance when this happens. Let $l_n: [0,t] \longrightarrow [0,\infty)$ be a mapping such that 
\[
l_n(s) =
\begin{cases}
\frac{\sigma_n^0}{\tau^0}(s) & \text{if } \quad 0\leq s \leq \tau^0 \\
\frac{\sigma_n^1}{\tau^1}(s-\tau^0)+\sigma_n^0 & \text{if } \quad \tau^0\leq s\leq \tau^0 + \tau^1 \\
\frac{\sigma_n^2}{\tau^2}\left(s-\sum \limits_{i=0}^{1}\tau ^i\right)+\sum \limits_{i=0}^{1}\sigma_n^i & \text{if } \quad \sum \limits_{i=0}^{1}\tau ^i \leq s \leq \sum \limits_{i=0}^{2}\tau ^i\\
\vdots \\
\frac{\sum \limits_{i=0}^{k-1}\sigma_n^i-t}{\sum \limits_{i=0}^{k-1}\tau^i-t}(s-t) +t  & \text{if } \quad\sum \limits_{i=0}^{k-1}\tau ^i \leq s \leq t.\\
\end{cases}
\]

Clearly, $l_n$ is a continuous and increasing piecewise function. However, we are not guaranteed that $l_n \in \Lambda_t.$ Since, for every $j=0,1,2,\dots$ $\sigma_n^j \longrightarrow \tau^j.$ We know that given a small $\epsilon>0$ there exists $N_j$ such that for every $n \geq N_j,$\\
$$\sum_{i=0}^{j}\sigma_n^i \in \left( \sum_{i=0}^{j}\tau^i - \epsilon, \sum_{i=0}^{j}\tau^i + \epsilon \right) \subset [0,t].$$ \\
For every $n\geq N,$ $l_n \in \Lambda_t$ with
\[
N= \max_{0\leq j \leq k-1} N_j.
\]
For each $s\leq t,$ because $l_n(s) - s$ is a piecewise linear function, its maximum absolute value is achieved at the boundary points of its intervals:
\[
|s-l_n(s)| \leq \max_{0\leq j \leq k-1} \left| \sum_{i=0}^{j}\tau^i - \sum_{i=0}^{j}\sigma_n^i \right|.
\]
By a similar argument, given $m>0$ fixed and arbitrary, there exists $N^*$ such that for every $n \geq N^*$ and $s\leq t$,

\[
|s-l_n(s)| \leq \max_{0\leq j \leq k-1} \left| \sum_{i=0}^{j}\tau^i - \sum_{i=0}^{j}\sigma_n^i \right| < \frac{1}{m}.
\]
It follows that
\[
\sup_{s \leq t} |s- l_n(s)| \leq \frac{1}{m}.
\]

Denote $N_m = \max\{N, N^*\}.$ There exists a minimal $k$ such that $k > N_m$. Letting $m=m_1$ and $k=k_1$, then we define the element $\lambda _1 = l_{k_1} \in \Lambda_t.$ Next taking $m_2>m_1$, there exists a minimal $k_2 > k_1$ such that $\lambda_2 = l_{k_2} \in \Lambda_t$ and 
\[
\sup_{s \leq t} |s- l_{k_2}(s)| \leq \frac{1}{m_2}.
\]
For the elements $l_{k_n}=\lambda_n \in \Lambda_t$ we have that since $m_n \longrightarrow 0$ then 
\[
\sup_{s \leq t} |s- \lambda_n(s)| \longrightarrow 0.
\]
Note that if $t<\tau^0$ then we can consider the elements $\lambda_n = l_n|_{[0,t]} \in \Lambda_t$ such that $n\geq N^*,$ which also implies

\[
\sup_{s \leq t} |s- \lambda_n(s)| \longrightarrow 0.
\]
The only thing that remains to be shown is that 
\[
\sup_{s \leq t} d((f_n\star f^1_n\star f^2_n\star \dots)(\lambda_ns), (f\star f^1\star f^2\star \dots)(s))\longrightarrow 0.
\] 
However, by the product space convergence, $f_n\longrightarrow f, f_n^k\longrightarrow f^k, \dots$ and $\sigma^0_n\longrightarrow \tau^0, \sigma_n^k\longrightarrow \tau^k,\dots$ in the Skorohod topology. Because the limit functions $f, f^1, \dots$ are continuous by hypothesis, convergence in the Skorohod topology implies uniform convergence on compact sets. Therefore, for every $s \leq t$:
\[
d((f_n\star f^1_n\star f^2_n\star \dots)(\lambda_ns), (f\star f^1\star f^2\star \dots)(s)) =
\begin{cases}
d(f_n(\lambda_ns), f(s)) & \text{if } \quad 0\leq s\leq\tau^0 \\
d(f_n^1(\lambda_ns), f^1(s)) & \text{if } \quad \tau^0\leq s \leq \tau^0 + \tau^1 \\
\vdots \\
d(f_n^k(\lambda_ns), f^k(s))  & \text{if } \quad\sum \limits_{i=0}^{k-1}\tau ^i \leq s \leq t\\
\end{cases}
\]
and 
\[
\sup_{s \leq t} |s- \lambda_n(s)| \longrightarrow 0.
\]
So, 
\[
\sup_{s \leq t} d((f_n\star f^1_n\star f^2_n\star \dots)(\lambda_ns), (f\star f^1\star f^2\star \dots)(s))\longrightarrow 0
\]
We therefore have the convergence $(f_n\star f^1_n\star f^2_n\star \dots)\longrightarrow (f\star f^1\star f^2\star \dots)$ in the Skorohod topology.  
\end{proof}

\begin{proof}[Proof of Theorem \ref{thmscaling}]
Let 
\[Z^n = \left( \frac{1}{n} Z_n^{(0,\theta)}(\lfloor \alpha_n t \rfloor), t \ge 0\right)\]
and fix $x\in (0,1)$.  By Theorem \ref{thm:censored}, we may assume that $Z_n^{(0,\theta)}= (\tilde Z_n^{(0,\theta)}(\beta^n_k))_{k\geq 0}$.  Suppose that $Z^n(0)=\tilde Z^n(0)\Rightarrow x$.  Using the excursion decomposition of $\tilde Z_n^{(0,\theta)}$ around $0$ \cite{kallenberg}, we can decompose $\tilde Z^n$ in a vector of excursions $(\tilde Y^n_0,\tilde Y^n_1,\dots)$ such that
\[ \tilde Z^n = \tilde Y^n_0 \star \tilde Y^n_1\star \cdots\]
and $\tilde Y^n_0,\tilde Y^n_1,\dots$ are independent with $\tilde Y^n_1,\tilde Y^n_2,\dots$ being iid with $n \tilde Y^n_1(0) \sim\upsilon_n^{(0,\theta)}$.  Let $\tau^n_i = \inf\{t: \tilde Y^n_i(t)=0\}$.  By Proposition \ref{prop:jump_in_dist}, for $i\geq 1$, $\tilde Y^n_i(0)\Rightarrow V$ where $V$ has density $p(x)$.  Thus Theorem \ref{thm:hitting_time_convergence} implies that for all $i$, $(\tilde Y^n_i,\tau^n_i) \Rightarrow (\hat Z, \tau)$, where for $i\geq 1$ $\hat Z(0)\sim V$ and for $i=0$, $\hat Z(0)=x$.  Since the $((\tilde Y^n_i,\tau^n_i))_{i\geq 0}$ are independent, we have that 
\[((\tilde Y^n_0,\tau^n_0),(\tilde Y^n_1,\tau^n_1),\dots) \Rightarrow ((\hat Z_0, \tau_0),(\hat Z_1, \tau_1),\dots)\]
where the $((\hat Z_i,\tau_i))_{i\geq 0}$ are independent.  By the Skorokhod representation theorem, we may assume that this convergence is happening almost surely.  Since the $\tau_i$ are independent, we clearly have that $\sum_i\tau_i=\infty$ almost surely.  Applying Theorem \ref{thm:concatenation}, we see that 
\[ \tilde Z^n \to_{as} \hat Z_0\star \hat Z_1\star\cdots := Z.\]
Note that since there exists $\epsilon>0$ such that $\P(\tau_i>\epsilon)>0$ for all $i$, for any fixed time $T$ there is a random finite constant $K$ such that for $n$ sufficiently large, there are at most $K$ excursions of $\tilde Z^n$ end before time $T$, which means that 
\[ \max_{i\leq \alpha_n T} |\beta^n_i -i| < K  \quad \textrm{ since } \quad \beta^n_i = i - |\{k \leq i: \tilde Z^{(0,\theta)}(k) =0\}|.\]
Therefore  $\max_{i\leq \alpha_n T} |\beta^n_i/n -i/n|\to 0$ and it follows from Theorem \ref{thm:censored} that $Z^n\to_{as} Z$.  The identification of generator (and its domain) of $Z$ follows from the construction on \cite[p. 65]{Mandl1968}.
\end{proof}

\bibliographystyle{plain}
\bibliography{bib}
\end{document}